\documentclass[a4paper,11pt]{amsart}

\usepackage[utf8]{inputenc}		
\usepackage[T1]{fontenc}
\usepackage[english]{babel}

\usepackage{tikz,tikz-3dplot, forloop}

\usepackage{amsfonts}			
\usepackage{amsmath}
\usepackage{amssymb}
\usepackage{amsthm}

\usepackage{graphicx}			
\usepackage{tikz}
	\usetikzlibrary{cd}

\usepackage{hyperref}			
	\hypersetup{colorlinks=false, urlcolor=black, linkcolor=black}

\usepackage{enumitem}			

\usepackage{color}				
\usepackage{float}

\usepackage{mathrsfs}

\newcommand{\Z}{\mathbb{Z}}						
\newcommand{\R}{\mathbb{R}}						
\newcommand{\C}{\mathbb{C}}						

\newcommand{\B}{\mathbb{B}}

\newcommand{\eps}{\varepsilon}					

\newcommand{\dd}								
	{\mathop{}\!\mathrm{d}}						
\newcommand{\ddn}[1]							
	{\mathop{}\!\mathrm{d^{#1}}}

\newcommand{\abs}[1]							
	{\left| #1 \right|}
\newcommand{\smallabs}[1]						
	{\lvert #1 \rvert}	
\newcommand{\norm}[1]							
	{\left\lVert #1 \right\rVert}	
\newcommand{\smallnorm}[1]						
	{\lVert #1 \rVert}						
\newcommand{\ip}[2]								
	{\left< #1 , #2 \right>}

\DeclareMathOperator{\intr}{int}				
				
\DeclareMathOperator{\capac}{Cap}			

\newcommand{\loc}{\mathrm{loc}}

\newcommand{\qrval}{\Sigma}

\newcommand{\cH}{\mathcal{H}}

\newtheorem{thm}{Theorem}[section]{\bf}{\it}
\newtheorem{lemma}[thm]{Lemma}

\newtheorem{cor}[thm]{Corollary}

\newtheorem{innernamedthm}{}[section]{\bf}{\it}
\newenvironment{namedthm}[1]
{\renewcommand\theinnernamedthm{#1}\innernamedthm}
{\endinnercustomthm}

\newtheorem{innercustomthm}{Theorem}[section]{\bf}{\it}

{\bf}{\it}

{\bf}{\it}

\theoremstyle{definition}
\newtheorem{defn}[thm]{Definition}
\newtheorem{ex}[thm]{Example}

\theoremstyle{remark}

\numberwithin{equation}{section}

\begin{document}

\title{A single-point Reshetnyak's theorem}

\author[I. Kangasniemi]{Ilmari Kangasniemi}
\address{Department of Mathematics, Syracuse University, Syracuse,
NY 13244, USA }
\email{kikangas@syr.edu}

\author[J. Onninen]{Jani Onninen}
\address{Department of Mathematics, Syracuse University, Syracuse,
NY 13244, USA and  Department of Mathematics and Statistics, P.O.Box 35 (MaD) FI-40014 University of Jyv\"askyl\"a, Finland
}
\email{jkonnine@syr.edu}

\date{\today}

\maketitle

\begin{center}
{\it Dedicated to the memory of Yurii Reshetnyak}
\end{center} 

\begin{abstract}
	We prove a single-value version of Reshetnyak's theorem. Namely, if a non-constant  map  $f \in W^{1,n}_{\text{loc}}(\Omega, \mathbb{R}^n)$ from a domain $\Omega \subset \mathbb{R}^n$ satisfies the estimate $\lvert Df(x) \rvert^n \leq K J_f(x) + \Sigma(x) \lvert f(x) - y_0 \rvert^n $ for some $K \geq 1$, $y_0\in \mathbb{R}^n$ and $\Sigma \in L^{1+\varepsilon}_{\text{loc}}(\Omega)$, then $f^{-1}\{y_0\}$ is discrete, the local index $i(x, f)$ is positive in $f^{-1}\{y_0\}$, and every neighborhood of a point of $f^{-1}\{y_0\}$ is mapped to a neighborhood of $y_0$. Assuming this estimate for a fixed $K$ at every $y_0 \in \mathbb{R}^n$ is equivalent to assuming that the map $f$ is $K$-quasiregular, even if the choice of $\Sigma$ is different for each $y_0$. Since the estimate also yields a single-value Liouville theorem, it hence appears to be a good pointwise definition of $K$-quasiregularity. As a corollary of our single-value Reshetnyak's theorem, we obtain a higher-dimensional version of the argument principle that played a key part  in the solution to the Calder\'on problem.
\end{abstract}

\section{Introduction}

For a given domain (i.e.\ an open connected set)  $\Omega \subset \R^n$ with $n \geq 2$, a mapping $f \colon \Omega \to \R^n$ is called \emph{$K$-quasiregular} for $K \geq 1$ if $f$ is in the local Sobolev space $W^{1,n}_\loc(\Omega, \R^n)$ and
\begin{equation}\label{eq:qr-def}
	\abs{Df(x)}^n \leq K J_f(x)
\end{equation}
for almost every (a.e.) $x \in \Omega$. Here, $\abs{Df(x)}$ denotes the operator norm of the weak derivative $Df(x)$ at $x$, and $J_f(x) = \det Df(x)$ is the Jacobian determinant. A mapping $f$ is then called \emph{quasiregular} if it is $K$-quasiregular for some $K \geq 1$. The synonymous term \emph{mapping of bounded distortion} is also used in the literature~\cite{Reshetnyak-book}. 

Despite the assumptions being entirely analytic, the distortion inequality \eqref{eq:qr-def} implies multiple topological regularity properties for quasiregular maps. For instance, every quasiregular map has a continuous representative \cite{Reshetnyak_continuity}. This for instance follows from the fact that quasiregular mappings belong to a higher
Sobolev class $W^{1,n+\varepsilon}_\loc(\Omega, \R^n)$ than initially assumed, by use of Gehring's lemma and reverse H\"older inequalities~\cite{Gehring, Iwaniec-GehringLemma,  Meyers-Elcrat_HigherInt}.

The most fundamental topological consequence of \eqref{eq:qr-def} is, however, a deep result of Reshetnyak~\cite{Reshetnyak_QROrigin, Reshetnyak_Theorem2}.

\begin{namedthm}{Reshetnyak's theorem}
	A non-constant quasiregular map is open, discrete, and sense preserving.
\end{namedthm}

Here, a continuous map $f \colon \Omega \to \R^n$ is \emph{open} if $f(U)$ is an open set for every open $U \subset \Omega$, \emph{discrete} if $f^{-1} \{y\}$ is a discrete subset of $\Omega$ for every $y \in \R^n$, and \emph{sense-preserving} if $f$ has locally positive topological degrees. In particular, the discreteness of $f$ allows for the definition of a local topological index $i(x, f) \in \Z$ at every $x \in \Omega$, and the sense-preserving part then states that $i(x,f) > 0$. Thus, quasiregular
mappings are generalized branched coverings with bounded distortion.  Reshetnyak's work permitted the geometric methods of modulus and curve families to be used to great effect in building a theory analogous to that of analytic functions in the complex plane; see the monographs~\cite{Iwaniec-Martin_book, Reshetnyak-book, Rickman_book, Vuorinen_book}.

On the other hand, in \cite[Section 8.5]{Astala-Iwaniec-Martin_Book}, Astala, Iwaniec and Martin introduced a generalization of \eqref{eq:qr-def} of the form 
\begin{equation}\label{eq:qrvalue0}
	\abs{Df(x)}^n \leq K J_f(x) + \qrval(x) \abs{f(x)}^n,
\end{equation}
where $\qrval$ is a locally integrable function. This is a higher dimensional version of the planar Beltrami-type equation
\begin{equation}\label{eq:qrvalue0in2D}
	\partial_{\overline{z}} f(z) = \mu(z) \partial_z f(z) + A(z) f(z),
\end{equation}
where $f, \mu, A \colon \Omega \to \C, \Omega \subset \C$, are complex functions with $\norm{\mu}_{L^\infty(\Omega)} < 1$ and $A \in L^2_\loc(\Omega)$. Solutions of \eqref{eq:qrvalue0in2D} have been studied in e.g.\ \cite{Vekua_generalized-analytic}, and are connected to the pseudoanalytic functions of Bers \cite{Bers_pseudoanalytic}. The 2D solutions have already played a key part in various important results, such as the solution to the Calder\'on problem in \cite{Astala-Paivarinta}.

Relying on powerful 2D existence theorems and the ideas presented in \cite[Section III]{Vekua_generalized-analytic}, Astala and P\"aiv\"arinta  gave a Liouville-type theorem for entire planar solutions of \eqref{eq:qrvalue0in2D} that vanish at infinity, under the assumption that $A$ is compactly supported; see \cite[Proposition 3.3 a)]{Astala-Paivarinta}. Astala, Iwaniec and Martin then gave a version of this result for non-compactly supported $A$ in \cite[Theorem 8.5.1]{Astala-Iwaniec-Martin_Book}, and conjectured that the same holds in higher dimensions for solutions of \eqref{eq:qrvalue0}. This conjecture was recently resolved by the authors in \cite{Kangasniemi-Onninen_Heterogeneous}. In our proof of the conjecture, we referred to the generalized distortion inequality~\eqref{eq:qrvalue0} as a ``heterogeneous distortion inequality''. However, based on the results we show in this paper, the following term is more appropriate.

\begin{defn}\label{def:poinwiseqr}
	Let $\Omega \subset \R^n$ be a  domain, let $y_0 \in \R^n$, let $K \geq 1$, and let $\qrval \in L^{1+\eps}_\loc(\Omega)$ for some $\eps > 0$. Suppose that $f \in W^{1,n}_\loc(\Omega, \R^n)$. Then we say that $f$ has a \emph{$(K, \qrval)$-quasiregular value at $y_0$} if
	\begin{equation}\label{eq:qrvalue}
		\abs{Df(x)}^n \leq K J_f(x) + \qrval(x) \abs{f(x) - y_0}^n
	\end{equation}
	for a.e.\ $x \in \Omega$.
\end{defn}

Note that we assume a tiny amount of higher integrability of $\qrval$. It was shown in \cite[Theorem 1.1]{Kangasniemi-Onninen_Heterogeneous} that if $f$ has a $(K, \qrval)$-quasiregular value, then this higher integrability of $\qrval$ implies that $f$ is locally H\"older continuous in $\Omega$. Heuristically, maps $f$ satisfying \eqref{eq:qrvalue} are restricted similarly to quasiregular maps when $f(x)$ is close to $y_0$, but may behave more like arbitrary $W^{1, n+n\eps}_\loc(\Omega, \R^n)$-maps when $f(x)$ is away from $y_0$. It is also noteworthy that \eqref{eq:qrvalue} still allows for $\Omega$ to have regions where $J_f$ is negative, or regions where $J_f \equiv 0$ while $Df \neq 0$.

Our main motivation for this term is the following theorem, which is the main result of this paper.
\begin{thm}\label{thm:one_point_Reshetnyak}
	Let $\Omega \subset \R^n$ be a  domain, and let $f \in W^{1,n}_\loc(\Omega, \R^n)$. Suppose that $f$ has a $(K, \qrval)$-quasiregular value at $y_0 \in \R^n$. Then either $f \equiv  y_0$ or the following conditions hold:
	\begin{itemize}
		\item $f^{-1} \{y_0\}$ is a discrete subset of $\Omega$,
		\item at every $x_0 \in f^{-1} \{y_0\}$ the local index $i(x_0, f)$ is positive, and
		\item for every neighborhood $U$ of a point $x_0 \in f^{-1} \{y_0\}$, we have $y_0 \in \intr f(U)$.
	\end{itemize}
\end{thm}

That is, the condition \eqref{eq:qrvalue} with $\qrval \in L^{1+\eps}_\loc(\Omega)$ implies Reshetnyak's theorem at the pre-images of the single point $y_0$. To our knowledge, this is the first point-wise version of Reshetnyak's theorem in the quasiregular literature. Note that the assumption of higher integrability of $\qrval$ is mandatory; see Example \ref{ex:higher_int_needed}. We also emphasize that maps satisfying \eqref{eq:qrvalue} are not necessarily locally quasiregular in a neighborhood of $f^{-1} \{y_0\}$; see example \ref{ex:not_local_qr}.

Our choice of terminology suggests a connection between quasiregular maps and maps that have a quasiregular value at every $y_0 \in \R^n$. This is provided by the following theorem.

\begin{thm}\label{thm:equivalence}
	Let $\Omega \subset \R^n$ be a domain, let $f \in W^{1,n}_\loc(\Omega, \R^n)$, and let $K \geq 1$. Then the following are equivalent.
	\begin{itemize}
		\item The map $f$ is $K$-quasiregular.
		\item For every $y_0 \in \R^n$, there exists $\qrval_{y_0} \in L^{1+\eps}_\loc(\Omega)$ such that $f$ has a $(K, \qrval_{y_0})$-quasiregular value at $y_0$.
	\end{itemize}
\end{thm}

Recall that the \emph{Liouville theorem} for quasiregular maps asserts that a bounded quasiregular mapping in $\R^n$ is constant~\cite{Reshetnyak_Liouville}.
To further motivate our terminology, we restate \cite[Theorem 1.2]{Kangasniemi-Onninen_Heterogeneous} in a way which clearly shows that it is indeed a single-value version of the Liouville theorem. 
\begin{thm}[{\cite[Theorem 1.2]{Kangasniemi-Onninen_Heterogeneous}}]\label{thm:Liouville_restated}
	Let $f \in W^{1,n}_\loc(\R^n, \R^n)$, and suppose that $f$ is bounded. If $f$ has a $(K, \qrval)$-quasiregular value at $y_0 \in \R^n$ where $\qrval \in L^{1}(\R^n) \cap L^{1+\eps}_\loc(\R^n)$, then either $f \equiv y_0$ or $y_0 \notin f(\R^n)$.
\end{thm}

Moreover, we recall that if $f \colon \Omega \to \R^n$ is $K$-quasiregular with $K < 1$, then $f$ is necessarily constant. This result too has a single-valued counterpart. 

\begin{thm}\label{thm:small_K}
	Let $\Omega \subset \R^n$ be a  domain, and let $f \in W^{1,n}_\loc(\Omega, \R^n)$. Suppose that $f$ has a $(K, \qrval)$-quasiregular value at $y_0 \in \R^n$, where instead of $K \geq 1$ we assume that $0 \leq K < 1$. Then either $f \equiv y_0$ or $y_0 \notin f(\Omega)$.
\end{thm}

We note that if $f \colon \Omega \to \R^n$ satisfies \eqref{eq:qrvalue} with $K$, $\qrval$, and $y_0$, and if $\iota \colon \R^n \to \R^n$ is a reflection along one coordinate axis, then $\iota \circ f$ satisfies \eqref{eq:qrvalue} with $-K$, $\qrval$, and $\iota(y_0)$. Hence, using negative values of $K$ in Definition \ref{def:poinwiseqr} would lead to orientation-reversing versions of Theorems \ref{thm:one_point_Reshetnyak}--\ref{thm:small_K}.

It is also illuminating to consider what the condition \eqref{eq:qrvalue} looks like for the most natural choice of $\qrval \in L^{1+\eps}_\loc(\Omega)$ with $\eps >0$. That is, we select a point $x_0 \in f^{-1} \{y_0\}$, and define $\qrval(x) = C \abs{x-x_0}^{-q}$ for some $C > 0$ and $q < n$. This leads to \eqref{eq:qrvalue} taking the form
\[
	\abs{Df(x)}^n \leq K J_f(x) + C \left( \frac{\abs{f(x) - f(x_0)}}{\abs{x - x_0}^\gamma} \right)^n, \quad \text{where }\gamma \in (0, 1).
\]

\subsection{Argument principle}

In the solution of the Calder\'on problem by Astala and P\"aiv\"arinta in~\cite{Astala-Paivarinta}, another key tool besides the Liouville theorem is a version of the argument principle for solutions of \eqref{eq:qrvalue0in2D}. Using the terminology we have introduced, the statement they show in \cite[Proposition~3.3~b)]{Astala-Paivarinta} is as follows: \emph{Let $f \in W^{1,p}_\loc(\C, \C)$ for some $p > 2$ be such that $f$ has a $(K, \qrval)$-quasiregular value at $0$. Suppose that there exists $\lambda \in \C \setminus \{0\}$ such that 
\begin{equation}\label{eq:identityatinfinity}
	\frac{\abs{f(z) - \lambda z}} { \abs{z} }\to 0, \quad  \textnormal{ as } z \to \infty .
\end{equation}
Then $f(z) = 0$ at exactly one point $z \in \C$.}

The proofs of this result in \cite{Astala-Paivarinta} and in \cite[Section 18.5]{Astala-Iwaniec-Martin_Book} rely heavily on arguments that are specific to two dimensions. Nevertheless, by combining Theorem \ref{thm:one_point_Reshetnyak} with a standard degree theory argument, we immediately  obtain a higher dimensional version of this result with far more general assumptions.

\begin{cor}\label{cor:argument_principle}
	Let $f_1, f_2 \in W^{1,n}_\loc(\R^n, \R^n)$ be such that both $f_i$ have a $(K_i, \qrval_i)$-quasiregular value at $y_0 \in \R^n$. Suppose that
	\[
		\liminf_{x \to \infty} \abs{f_2(x) - y_0} \neq 0 
		\quad \text{and} \quad
		\lim_{x \to \infty} \abs{f_1(x) - f_2(x)} = 0.
	\]
	Then
	\[
		\sum_{x \in f_1^{-1}\{y_0\}} i(x, f_1) = \sum_{x \in f_2^{-1}\{y_0\}} i(x, f_2).
	\]
	In particular, if $f_2^{-1}\{y_0\}$ is a singleton, then $f_1^{-1}\{y_0\}$ is also a singleton.
\end{cor} 

Note that Corollary~\ref{cor:argument_principle} allows us to replace the condition~\eqref{eq:identityatinfinity} in the result of Astala and P\"aiv\"arinta by just requiring that $\abs{f(z) - \lambda z} \to 0$ as $z \to \infty$.

\subsection{Structure of this paper}

Sections~\ref{sec:disconnected_preimage} to~\ref{sec:finishing_proof} comprise the proof of Theorem~\ref{thm:one_point_Reshetnyak}. In Section \ref{sec:disconnected_preimage}, we prove that if a map $f$ with a quasiregular value at $y_0$ is not constant, then the set $f^{-1}\{y_0\}$ has Hausdorff $1$-measure zero, which in turn implies that $f^{-1}\{y_0\}$ is totally disconnected. In Section~\ref{sec:degree_theory}, we recall several key lemmas of topological degree theory for maps with a totally disconnected fiber.

Section~\ref{sec:value_degree} then contains the proof of the key part of Theorem~\ref{thm:one_point_Reshetnyak}, which is that the degree of $f$ turns positive-valued when sufficiently close to $f^{-1}\{y_0\}$. At the heart of this part of the proof is a local version of the proof of the Liouville Theorem \ref{thm:Liouville_restated} in \cite{Kangasniemi-Onninen_Heterogeneous}. Afterwards, the proof of Theorem~\ref{thm:one_point_Reshetnyak} is completed in Section~\ref{sec:finishing_proof}.

Section~\ref{sec_other_proofs} contains the proofs of the remaining results from the introduction; Theorem~\ref{thm:equivalence}, Theorem~\ref{thm:small_K}, and Corollary~\ref{cor:argument_principle}. Finally, in Section~\ref{sect:examples}, we present several examples outlining the possible behavior of maps with K-quasiregular values.

\section{Disconnected preimage}\label{sec:disconnected_preimage}

The first step in the proof of Theorem \ref{thm:one_point_Reshetnyak} is to show that $f^{-1} \{y_0\}$ is disconnected. The path to this is by showing that  $\cH^1(f^{-1} \{y_0\}) = 0$, by using a version of the argument in \cite{Onninen-Zhong_MFD-loglog-proof}.

\begin{lemma}\label{lem:trunc_loglog_bound}
	Let $\Omega \subset \R^n$, let $y_0 \in \R^n$, and let $f \in W^{1,n}_\loc(\Omega, \R^n)$ satisfy \eqref{eq:qrvalue} with $K \in [1, \infty)$ and $\qrval \in L^{p}_\loc(\Omega)$ for some $p > 1$. Then there exists a neighborhood $V$ of $y_0$ such that
	\begin{equation}\label{eq:uniform_grad_int_bound}
		\int_D \big\lvert\nabla \min(\log \log \abs{f - y_0}^{-1}, k)\big\rvert^n \leq C_D < \infty
	\end{equation}
	for all $k \in \Z_{> 0}$ and all domains $D$ compactly contained in $f^{-1} V$, where $C_D = C_D(D, n, K, \qrval)$ is independent on $k$.
\end{lemma}
\begin{proof}
	We define $U = f^{-1} \B^n(y_0, 1/e)$, which is open since $f$ is continuous by \cite[Theorem 1.1]{Kangasniemi-Onninen_Heterogeneous}. We denote $u_k = \min(\log \log \abs{f-y_0}^{-1}, k)$ for $k \in \Z_{>0}$, and note that $u_k \in W^{1,n}_\loc(U)$ and $u_k \geq 0$. Let $D$ then be a domain that's compactly contained in $U$. We select $\eta \in C^\infty_0(U, [0, 1])$ such that $\eta \equiv 1$ on $D$.
	
	Following the methods in \cite{Onninen-Zhong_MFD-loglog-proof}, we apply \cite[Lemma 6.2]{Kangasniemi-Onninen_Heterogeneous} (which in turn is a version of \cite[Lemma 2.1]{Onninen-Zhong_MFD-loglog-proof} with slightly more general assumptions) on $f - y_0$ with 
	\[
		\Psi(t) = \frac{1}{2t^\frac{n}{2}} \int_0^{\min(t, e^{-1})} \frac{\varphi_\eps(s)}{s \log^{n}(s^{-1})} \dd s, 
		\qquad \text{where }
		\varphi_\eps(s) = \frac{1}{1 + \eps 2^{s^{-1}}},
	\]
	and repeat the estimates in \cite[(11)-(12)]{Onninen-Zhong_MFD-loglog-proof}. Since $\abs{f - y_0} < e^{-1}$ in $U$, the result is that
	\begin{multline}\label{eq:loglog_est_1}
		\int_U \eta^n \frac{J_f}{\abs{f - y_0}^n \log^n \bigl(\abs{f - y_0}^{-2}\bigr)} \varphi_\eps(\abs{f - y_0}^{-2})\\
		\leq C(n) \int_U \abs{\nabla \eta} \eta^{n-1} \frac{\abs{Df}^{n-1}}{\abs{f - y_0}^{n-1} \log^{n-1} \abs{f - y_0}^{-2}} \varphi_\eps^{\frac{n-1}{n}}(\abs{f - y_0}^{-2}).
	\end{multline}
	
	For any $a, b \geq 0$ and $\delta > 0$, Young's inequality implies that
	\[
		a^{n-1} b
		= \frac{a^{n-1}}{\delta^{\frac{n-1}{n}}} (b \delta^{\frac{n-1}{n}})
		\leq \frac{n-1}{n} \frac{a^{n}}{\delta} + \frac{1}{n} b^n \delta^{n-1}.
	\]
	We apply an estimate of this type on the right hand side of \eqref{eq:loglog_est_1}, and consequently obtain that
	\begin{multline}\label{eq:loglog_est_2}
		\int_U \eta^n \frac{J_f}{\abs{f - y_0}^n \log^n \bigl(\abs{f - y_0}^{-2}\bigr)} \varphi_\eps(\abs{f - y_0}^{-2})\\
		\leq 
		\frac{n-1}{nK} \int_U \eta^{n} \frac{\abs{Df}^{n}}{\abs{f - y_0}^{n} \log^{n} \abs{f - y_0}^{-2}} \varphi_\eps(\abs{f - y_0}^{-2})\\
		+ \frac{C^n(n) K^{n-1}}{n} \int_U \abs{\nabla \eta}^n.
	\end{multline}
	We combine \eqref{eq:loglog_est_2} with the assumed pointwise estimate \eqref{eq:qrvalue}, and absorb the term with $\abs{Df}^n$ to the left hand side; this is possible since the function $\varphi_\eps$ ensures that the integral is finite. The result of this is that
	\begin{multline}\label{eq:loglog_est_3}
		\int_U \eta^{n} \frac{\abs{Df}^{n}}{\abs{f - y_0}^{n} \log^{n} \abs{f - y_0}^{-2}} \varphi_\eps(\abs{f - y_0}^{-2})\\
		\leq C^n(n) K^{n} \int_U \abs{\nabla \eta}^n
		+ n \int_U \eta^{n} \frac{\qrval}{\log^{n} \abs{f - y_0}^{-2}} \varphi_\eps(\abs{f - y_0}^{-2}).
	\end{multline}
	Now, since \eqref{eq:loglog_est_3} no longer has any Jacobians in it,  all terms in the estimate are non-negative. Hence, we may let $\eps \to 0$ and use monotone convergence to obtain
	\begin{multline}\label{eq:loglog_est_4}
		\int_U \eta^{n} \frac{\abs{Df}^{n}}{\abs{f - y_0}^{n} \log^{n} \abs{f - y_0}^{-2}}\\
		\leq C^n(n) K^{n} \int_U \abs{\nabla \eta}^n
		+ n \int_U \eta^{n} \frac{\qrval}{\log^{n} \abs{f - y_0}^{-2}}\\
		\leq C^n(n) K^{n} \int_U \abs{\nabla \eta}^n
		+ \frac{n}{2^n} \int_U \eta^{n} \qrval.
	\end{multline}
	However, we have
	\[
		\abs{\nabla u_k} 
		\leq \frac{\abs{\nabla \abs{f-y_0}}}{\abs{f-y_0} \log \abs{f-y_0}^{-1}}
		\leq \frac{\abs{Df}}{\abs{f-y_0} \log \abs{f-y_0}^{-2}}
	\]
	a.e.\ in $U$, and the claim hence follows by applying \eqref{eq:loglog_est_4} along with the fact that $\eta \equiv 1$ on $D$.
\end{proof}

If $u_k = \min(\log \log \abs{f-y_0}^{-1}, k)$ as in the previous lemma, the result bounds the $L^n$-norms of $\abs{\nabla u_k}$ with a bound that doesn't depend on $k$. We then require a similar $k$-independent bound for the $L^n$-norms of $\abs{u_k}$. For this, we recall the following standard cutoff lemma and its proof.

\begin{lemma}\label{lem:loglog_sobolev}
	Let $B \subset \R^n$ be a ball, let $u \colon B \to [0, \infty]$ be measurable, and let $p \in [1, \infty)$. Suppose that for every $k \in \Z_{> 0}$, the function $u_k = \min(u, k)$ is in $W^{1,p}(B)$ and satisfies
	\[
		\int_B \abs{\nabla u_k}^p \leq C < \infty,
	\]
	 where $C$ is independent on $k$. Then either $u \equiv \infty$ a.e.\ on $B$, or $u \in W^{1,p}(B)$ and consequently
	 \begin{equation}\label{eq:uniform_int_bound}
	 	\int_B \abs{u_k}^p \leq C' < \infty,
	 \end{equation}
	 for all $k \in \Z_{> 0}$, with $C' = C'(u, B)$ independent of $k$.
\end{lemma}
\begin{proof}
	Suppose that $u$ is not identically $\infty$ on $B$. Then $u^{-1} [0, \infty)$ has positive measure, and therefore $u^{-1} [0, k_0)$ has positive measure for some $k_0 \in \Z_{>0}$. We first claim that $u \in L^1(B)$. Suppose to the contrary that the integral of $u$ over $B$ is infinite. If we denote by $(u_k)_B$ the average integral of $u_k$ over $B$, we hence have $(u_k)_B \to \infty$ monotonically as $k \to \infty$. Now, the Sobolev-Poincar\'e inequality and our assumption imply that
	\[
		\int_B \smallabs{u_k - (u_k)_B}
		\leq C_B \norm{\nabla u_k}_{L^n} \leq C_B C^\frac{1}{p} < \infty.
	\]
	On the other hand, for $k$ large enough that $(u_k)_B \geq k_0$, we also have
	\[
		\int_B \smallabs{u_k - (u_k)_B}
		\geq m_n(u^{-1} [0, k_0)) \cdot ((u_k)_B - k_0) \xrightarrow[k \to \infty]{} \infty.
	\]
	This is a contradiction, concluding the proof that $u \in L^1(B)$.
	
	Now, since $u \in L^1(B)$, we must have $u \neq \infty$ almost everywhere. Hence, $u_k$ form a Cauchy sequence in $W^{1,1}(B)$ and converge to $u$ a.e.\ in $B$ monotonely, which proves that $u$ is weakly differentiable. By our assumption on the integrals of $\abs{\nabla u_k}^n$ and monotone convergence, we get that $\norm{\nabla u}_{L^p} < \infty$, which by Sobolev embedding implies that $u \in W^{1,p}(B)$. We can hence pick $C' = \norm{u}_{L^p}^p < \infty$.
\end{proof}

We may now complete the proof of the main result of this section.

\begin{lemma}\label{lem:fiber_hausd_measure}
	Let $\Omega \subset \R^n$ be a connected domain, let $y_0 \in \R^n$, and let $f \in W^{1,n}_\loc(\Omega, \R^n)$ satisfy \eqref{eq:qrvalue} with $K \in [1, \infty)$ and $\qrval \in L^{p}_\loc(\Omega)$ for some $p > 1$. Then either $f \equiv y_0$, or $\cH^1(f^{-1}\{y_0\}) = 0$. In particular, in the latter case, $f^{-1}\{y_0\}$ is totally disconnected.
\end{lemma}
\begin{proof}
	By Lemmas \ref{lem:trunc_loglog_bound} and \ref{lem:loglog_sobolev}, we find a neighborhood $V$ of $y_0$ such that, for every ball $B$ compactly contained in $f^{-1} V$, we have either $f \equiv y_0$ on $B$ or $\log \log \abs{f - y_0}^{-1} \in W^{1,n}(B)$. Moreover, both of these clearly cannot hold on the same ball $B$, since $\log \log \abs{y_0 - y_0}^{-1} \equiv \infty$ is not integrable. Consequently, if $\log \log \abs{f - y_0}^{-1} \notin W^{1,n}_\loc(f^{-1} V)$, then there exists a connected component $U$ of $f^{-1} V$ such that $f \equiv y_0$ on $U$. We would then have by continuity that $f(\partial U \cap \Omega) \subset \{y_0\} \subset V$. Since the connected components of $f^{-1} V$ are open, this is only possible if $U = \Omega$.
	
	Hence, we either have $\log \log \abs{f - y_0}^{-1} \in W^{1,n}_\loc(f^{-1} V)$, or $f \equiv y_0$ on all of $\Omega$. We suppose that we're in the former case, and wish to prove that $\cH^{1}(f^{-1}\{y_0\}) = 0$. This will follow by a standard capacity argument.
	
	Indeed, we let $x \in f^{-1}\{y_0\}$, we let $B = \B^n(x, r)$ be such that $3B = \B^n(x, 3r)$ is compactly contained in $f^{-1} V$, we select $\eta \in C^\infty_0(3B)$ such that $\eta \geq 0$ and $\eta = 1$ on $2B$, and we define $v_k = k^{-1} \eta \min(\log \log \abs{f - y_0}^{-1}, k)$. Then every $v_k$ is a compactly supported $W^{1,n}$-function on $\B^n(x, r)$ such that $v_k \equiv 1$ in a neighborhood of $f^{-1}\{y_0\} \cap \overline{B}$. Hence, the functions $v_k$ are admissible for the $n$-capacity of the condenser $(f^{-1}\{y_0\} \cap \overline{B}, 3B)$ (for details, see e.g.\ \cite[pp. 27--28]{Heinonen-Kilpelainen-Martio_book}). Since the $L^n$-norms of $\nabla v_k$ tend to zero by \eqref{eq:uniform_grad_int_bound} and \eqref{eq:uniform_int_bound}, it follows that $\capac_n(f^{-1}\{y_0\} \cap \overline{B}, 3B) = 0$, and therefore we also have $\cH^{1}(f^{-1}\{y_0\}\cap \overline{B}) = 0$ by e.g.\ \cite[Theorem 2.6]{Heinonen-Kilpelainen-Martio_book}. The claim $\cH^{1}(f^{-1}\{y_0\}) = 0$ hence follows by considering a countable cover of $f^{-1}\{y_0\}$ by such $B$, and the proof is complete.
\end{proof}

\section{Degree theory}\label{sec:degree_theory}

Now, under the assumptions of Theorem \ref{thm:one_point_Reshetnyak}, we have due to Lemma \ref{lem:fiber_hausd_measure} that $f^{-1}\{y_0\}$ is totally disconnected. We next require that for every $x \in f^{-1}\{y_0\}$ there is a small enough $\eps > 0$ that the $x$-component of $f^{-1} \B^n(y_0, \eps)$ does not escape to the boundary.

\begin{lemma}\label{lem:compact_preimages}
	Let $\Omega \subset \R^n$ be a connected domain, let $y_0 \in \R^n$, and let $f \colon \Omega \to \R^n$ be a continuous map such that $f^{-1}\{y_0\}$ is totally disconnected. Then for every $x \in f^{-1}\{y_0\}$, there exists $\eps > 0$ such that if $U$ is the $x$-component of $f^{-1} \B^n(y_0, \eps)$, then $\overline{U}$ is a compact subset of $\Omega$.
\end{lemma}
\begin{proof}
	We may assume that $\Omega$ is bounded, since if $U$ is a non-empty connected component of $f^{-1} \B^n(y_0, \eps)$ and $\overline{U} \subset \Omega$, then enlarging the domain of definition $\Omega$ can not enlarge the component $U$.
	
	We then assume towards contradiction that there exists $x \in f^{-1}\{y_0\}$ such that, for every $\eps > 0$, the $x$-component $U_\eps$ of $f^{-1} \B^n(y_0, \eps)$ is not compactly contained in $\Omega$. Since $\Omega$ is bounded, this implies that $\overline{U_\eps} \cap \partial \Omega \neq \emptyset$. The sets $\overline{U_\eps}$ now form a descending sequence of compact connected sets. 
	
	It hence follows that the intersection $C = \bigcap_{\eps > 0} \overline{U_\eps}$ is compact and connected; see e.g.\ \cite[Theorem 6.1.19]{Engelking_topology}. We clearly have $x \in C$, and since $\overline{U_\eps} \cap \partial \Omega$ form a descending sequence of compact sets, we also muct have $C \cap \partial \Omega \neq \emptyset$. However, this is a contradiction, since now $C$ is a connected subset of $f^{-1}\{y_0\}$ that connects $x$ to the boundary of $\Omega$, yet $f^{-1}\{y_0\}$ is totally disconnected. The claim hence follows.
\end{proof}

With Lemma \ref{lem:compact_preimages}, we may hence reasonably start applying degree theory.

\begin{defn}
	Let $\Omega \subset \R^n$ be a connected domain, let $y_0 \in \R^n$, and let $f \colon \Omega \to \R^n$ be a continuous map such that $f^{-1}\{y_0\}$ is totally disconnected. We suppose that $U$ is a connected component of $f^{-1} \B^n(y_0, \eps)$ for some $\eps > 0$ such that $\overline{U}$ is a compact subset of $\Omega$. Note that if $x \in f^{-1}\{y_0\}$ and $U$ is the $x$-component of $f^{-1} \B^n(y_0, \eps)$ for a small enough $\eps$, then this property holds by Lemma \ref{lem:compact_preimages}. In particular, $U$ is an open set such that $f$ is well defined on $\overline{U}$ and $f\partial U \subset \partial \B^n(y_0, \eps)$.
	
	Hence, if $U$ is a component of $f^{-1} \B^n(y_0, \eps)$ with $\overline{U} \subset \Omega$ compact, then for every $y \notin f \partial U$ the map $f$ has a well-defined integer-valued topological degree $\deg(f, y, U)$ at $y$ with respect to $U$; see e.g.\ \cite[Section 1.2]{Fonseca-Gangbo-book}. In particular, this is true for all $y \in \B^n(y_0, \eps)$. Moreover, since $\B^n(y_0, \eps)$ is a connected set that doesn't meet $f \partial U$, the degree $\deg(f, y, U)$ is constant-valued on $\B^n(y_0, \eps)$; see e.g.\ \cite[Theorem 2.3 (3)]{Fonseca-Gangbo-book}. We call this constant value of $\deg(f, y, U)$ the \emph{degree of $f$ with respect to $U$}, and denote it $\deg(f, U)$.
\end{defn}

In case $f$ is a Sobolev map, we also have the following Jacobian formula for the degree.

\begin{lemma}\label{lem:sobolev_degree_formula}
	Let $\Omega \subset \R^n$ be a connected domain, let $y_0 \in \R^n$, and let $f \colon \Omega \to \R^n$ be a continuous map such that $f^{-1}\{y_0\}$ is totally disconnected. Let $U_\eps$ be a connected component of $f^{-1} \B^n(y_0, \eps)$ such that $\overline{U_\eps}$ is compactly contained in $\Omega$. Suppose that $f \in W^{1,n}_\loc(\Omega, \R^n)$. Then we have
	\[
		\deg(f, U_{\eps}) = \frac{1}{\omega_n \eps^n} \int_{U_{\eps}} J_f,
	\]
	where $\omega_n$ is the volume of a unit ball in $\R^n$.
\end{lemma}
\begin{proof}
	Since $f(\partial U_\eps) \subset \partial \B^n(y_0, \eps)$ which has measure zero, the claim follows e.g.\ from \cite[Proposition 5.25 and Remark 5.26 (ii)]{Fonseca-Gangbo-book}, which yield that
	\[
		\int_{U_{\eps}} J_f = \int_{\R^n} \deg (f, U_{\eps}, y) \dd y.
	\]
	Indeed, the integrand on the right hand side is $\deg(f, U_\eps)$ in $\B^n(y_0, \eps)$, and vanishes outside $\overline{\B^n(y_0, \eps)}$ due to e.g.\ \cite[Theorem 2.1]{Fonseca-Gangbo-book}.
\end{proof}

We point out the following useful property of pre-image sets that follows by a similar argument as Lemma \ref{lem:compact_preimages}

\begin{lemma}\label{lem:disjoint_preimages}
	Let $\Omega \subset \R^n$ be a connected domain, let $y_0 \in \R^n$, and let $f \colon \Omega \to \R^n$ be a continuous map such that $f^{-1}\{y_0\}$ is totally disconnected. Then for all $x_1, x_2 \in f^{-1}\{y_0\}$ such that $x_1 \neq x_2$, there exists $\eps > 0$ such that $x_1$ and $x_2$ are in different components of $f^{-1} \B^n(y_0, \eps)$.
\end{lemma}
\begin{proof}
	Suppose to the contrary that for all $\eps > 0$, the $x_1$-component $U_\eps$ of $f^{-1} \B^n(y_0, \eps)$ also contains $x_2$. If $\eps$ is small enough, we have by Lemma \ref{lem:compact_preimages} that $\overline{U_\eps}$ is a compact connected subset of $\Omega$. Hence, if we again define $C = \bigcap_{\eps > 0} \overline{U_\eps}$, then $C$ is connected by \cite[Theorem 6.1.19]{Engelking_topology}, $C \subset f^{-1} \{y_0\}$, and $\{x_1, x_2\} \subset C$. This contradicts the fact that $f^{-1}\{y_0\}$ is totally disconnected.
\end{proof}

We also note that by Lemmas \ref{lem:compact_preimages} and \ref{lem:disjoint_preimages}, we can get the pre-image components to be as small in measure as we want.

\begin{cor}\label{cor:small_preimages}
		Let $\Omega \subset \R^n$ be a connected domain, let $y_0 \in \R^n$, and let $f \colon \Omega \to \R^n$ be a continuous map such that $f^{-1}\{y_0\}$ is totally disconnected. Let $x_0 \in f^{-1}\{y_0\}$, and for all $\eps > 0$, let $U_\eps$ denote the $x_0$-component of $f^{-1} \B^n(y_0, \eps)$. Then $\lim_{\eps \to 0^+} m_n(U_\eps) = 0$.
\end{cor}
\begin{proof}
	The sets $U_\eps$ decrease as $\eps$ decreases. By Lemma \ref{lem:disjoint_preimages}, we have $\bigcap_{\eps > 0} U_\eps = \{x_0\}$, and by Lemma \ref{lem:compact_preimages} there is a $U_\eps$ with finite $m_n$-measure. Hence, basic convergence properties of measures imply the claim.
\end{proof}

\section{Positivity of the degree}\label{sec:value_degree}

\subsection{Negative values} Now, under the assumptions of Theorem \ref{thm:one_point_Reshetnyak}, if $f(x) = y_0$, then for small enough $\eps > 0$ we have a well-defined degree $\deg(f, U)$ for the $x$-component $U$ of $f^{-1} \B^n(y_0, \eps)$. Our next goal is to show that if $U$ is small enough, then this degree cannot be negative.

\begin{lemma}\label{lem:weakly_sense_preserving}
	Let $\Omega \subset \R^n$ be a bounded connected domain, let $y_0 \in \R^n$, and let $f \in W^{1,n}(\Omega, \R^n)$ be a non-constant map that satisfies \eqref{eq:qrvalue} for a.e.\ $x \in \Omega$, where $K \in [1, \infty)$ and $\qrval \in L^{p}(\Omega)$ for some $p > 1$. Suppose that $U$ is a non-empty component of $f^{-1}\B^n(y_0, \eps)$, where $\eps > 0$ is small enough that $\overline{U} \subset \Omega$. Then there exists $c = c(n, K, \qrval, \Omega) > 0$ such that $\deg(f, U) \geq 0$ if $m_n(U) < c$.
\end{lemma}
\begin{proof}
	Suppose that $\deg(f, U) < 0$. Then $\deg(f, U) \leq -1$, and the Jacobian formula for the degree from Lemma \ref{lem:sobolev_degree_formula} yields that
	\[
		\int_{U_\eps} J_f \leq - \omega_n \eps^n. 
	\]
	In particular, if we use $J_f^+$ and $J_f^-$ to denote the positive and negative part of the Jacobian, respectively, then we have
	\[
		\int_{U_\eps} J_f^- \geq \omega_n \eps^n.
	\]
	However, the distortion inequality \eqref{eq:qrvalue} can be rewritten as $\abs{Df}^n + K J_f^- \leq K J_f^+ + \abs{f - y_0}^n \qrval$. Since $J_f^+$ vanishes where $J_f^- > 0$, we hence have $J_f^- \leq K^{-1} \abs{f - y_0}^n \qrval$. Now we may estimate
	\begin{multline*}
		\omega_n \eps^n \leq \int_{U_\eps} J_f^- \leq \int_{U_\eps} \frac{\abs{f - y_0}^n \qrval}{K}
		\leq \frac{\eps^n}{K} \int_{U_\eps} \qrval
		\leq \frac{\eps^n}{K} \left[m_n(U_\eps)\right]^\frac{p-1}{p} \norm{\qrval}_{L^p}.
	\end{multline*}
	Hence, we may only have a negative $\deg(f, U)$ if
	\[
		m_n(U_\eps) \geq \left( \omega_n K/ \norm{\qrval}_{L^p} \right)^\frac{p}{p-1},
	\]
	where this lower bound is interpreted as $\infty$ if $\qrval$ is identically zero.
\end{proof}

\subsection{Zero values}

After this relatively short argument that negative values of $\deg(f, U)$ are impossible for small $U$, we next wish to similarly exclude the possibility of $\deg(f, U) = 0$ for small $U$. The proof of this is far more complicated; however, the ideas are in fact essentially a local version of the proof of \cite[Theorem 1.2]{Kangasniemi-Onninen_Heterogeneous}, and we can thankfully skip most of the laborious parts by directly re-using several of the lemmas in \cite{Kangasniemi-Onninen_Heterogeneous}.

The main trick that lets us access the methods of \cite{Kangasniemi-Onninen_Heterogeneous} is the following lemma.

\begin{lemma}\label{lem:sense_preserving_level_set_integral}
	Let $\Omega \subset \R^n$ be a bounded connected domain, let $y_0 \in \R^n$, and let $f \in W^{1,n}(\Omega, \R^n)$ be a non-constant map that satisfies \eqref{eq:qrvalue} for a.e.\ $x \in \Omega$, where $K \in [1, \infty)$ and $\qrval \in L^p(\Omega)$ for some $p > 1$. Suppose that $U$ is a non-empty component of $f^{-1}\B^n(y_0, \eps)$, where $\eps > 0$ is small enough that $\overline{U} \subset \Omega$ and $m_n(U) < c$ with $c = c(n, K, \qrval, \Omega) > 0$ given by Lemma \ref{lem:weakly_sense_preserving}. If $\deg(f, U) = 0$, then for every $r \in (0, \eps)$ we have
	\[
		\int_{U \cap \{x \in \R^n : \abs{f - y_0} < r\}} J_f = 0. 
	\]
\end{lemma}
\begin{proof}
	The set $U \cap \{x \in \R^n : \abs{f - y_0} < r\}$ is a disjoint union of components of $f^{-1}\B^n(y_0, r)$. We denote these components by $U_i$, where $i \in I$. Since $U_i \subset U$ and since $m_n(U) < c$, we have by Lemma \ref{lem:weakly_sense_preserving} that $\deg(f, U_i) \geq 0$ for every $i \in I$. Moreover, by using both parts of \cite[Theorem 2.7]{Fonseca-Gangbo-book}, we also have
	\[
		\sum_i \deg(f, U_i) = \deg(f, U) = 0.
	\]
	Hence, for every $i \in I$ we have $\deg(f, U_i) = 0$, and therefore by Lemma \ref{lem:sobolev_degree_formula} the integral of $J_f$ over every $U_i$ vanishes. The claim hence follows.
\end{proof}

The key part about Lemma \ref{lem:sense_preserving_level_set_integral} is that it allows us access to the argument of \cite[Lemma 5.3]{Kangasniemi-Onninen_Heterogeneous}. Since the original statement is only for globally defined $f \colon \R^n \to \R^n$, we recall the argument here.

\begin{lemma}\label{lem:sense_preserving_log_gradient_int}
	Let $\Omega \subset \R^n$ be a bounded connected domain, let $y_0 \in \R^n$, and let $f \in W^{1,n}(\Omega, \R^n)$ be a non-constant map that satisfies \eqref{eq:qrvalue} for a.e.\ $x \in \Omega$, where $K \in [1, \infty)$ and $\qrval \in L^p(\Omega)$ for some $p > 1$. Suppose that $U$ is a non-empty component of $f^{-1}\B^n(y_0, \eps)$, where $\eps > 0$ is small enough that $\overline{U} \subset \Omega$ and $m_n(U) < c$ with $c = c(n, K, \qrval, \Omega) > 0$ given by Lemma \ref{lem:weakly_sense_preserving}. If $\deg(f, U) = 0$, then we have
	\[
		\int_{U} \frac{\abs{Df}^n}{\abs{f-y_0}^n} < \infty
		\qquad \text{and} \qquad 
		\int_{U} \frac{J_f}{\abs{f-y_0}^n} = 0.
	\]
\end{lemma}
\begin{proof}
	Similarly as in the proof of Lemma \ref{lem:weakly_sense_preserving}, the equation \eqref{eq:qrvalue} implies that
	\[
		\int_U \frac{J_f^-}{\abs{f-y_0}^n} \leq \int_U \frac{\qrval}{K} < \infty.
	\]
	If we denote $U_r = U \cap f^{-1} \B^n(y_0, r)$ for all $r \in (0, \eps)$, then Lemma \ref{lem:sense_preserving_level_set_integral} yields that
	\[
		\int_{U_r} J_f^+ = \int_{U_r} J_f^-.
	\]
	By multiplying the above equation by $nr^{-n-1}$ and integrating, we get
	\[
		\int_0^\eps nr^{-n-1} \int_{U_r} J_f^+(x) \dd x \dd r = 
		\int_0^\eps nr^{-n-1} \int_{U_r} J_f^-(x) \dd x \dd r.
	\]
	Switching the order of integrals with Fubini-Tonelli yields
	\[
		\int_{U} J_f^+(x) \int_{\abs{f(x) - y_0}}^\eps nr^{-n-1} \dd r \dd x
		= \int_{U} J_f^-(x) \int_{\abs{f(x) - y_0}}^\eps nr^{-n-1} \dd r \dd x.
	\]
	By evaluating the inner integral, we have
	\[
		\int_{U} \left( \frac{J_f^+}{\abs{f-y_0}^n} - \frac{J_f^+}{\eps^n}\right)
		= \int_{U} \left( \frac{J_f^-}{\abs{f-y_0}^n} - \frac{J_f^-}{\eps^n}\right),
	\]
	and again using the fact that the integral of $J_f$ over $U$ is 0 yields
	\[
		\int_{U} \frac{J_f^+}{\abs{f-y_0}^n} = \int_{U} \frac{J_f^-}{\abs{f-y_0}^n} < \infty.
	\]
	We hence conclude that $J_f/\abs{f - y_0}^n$ is integrable over $U$ and has zero integral. Then \eqref{eq:qrvalue} immediately gives that $\abs{Df}^n/\abs{f - y_0}^n$ is integrable over $U$.
\end{proof}

We have now essentially merged with the proof of \cite[Theorem 1.2]{Kangasniemi-Onninen_Heterogeneous}. We now assemble the remaining pieces of the proof.

\begin{lemma}\label{lem:sense_preserving}
	Let $\Omega \subset \R^n$ be a bounded connected domain, let $y_0 \in \R^n$, and let $f \in W^{1,n}(\Omega, \R^n)$ be a non-constant map that satisfies \eqref{eq:qrvalue} for a.e.\ $x \in \Omega$, where $K \in [1, \infty)$ and $\qrval \in L^p(\Omega)$ for some $p > 1$. Suppose that $x_0 \in f^{-1}\{y_0\}$ and that $U$ is the $x_0$-component of $f^{-1}\B^n(y_0, \eps)$, where $\eps > 0$ is small enough that $\overline{U} \subset \Omega$ and $m_n(U) < c$ with $c = c(n, K, \qrval, \Omega) > 0$ given by Lemma \ref{lem:weakly_sense_preserving}. Then $\deg(f, U) > 0$.
\end{lemma}
\begin{proof}
	We suppose towards contradiction that $\deg(f, U) = 0$. Lemma \ref{lem:sense_preserving_log_gradient_int} then implies that $\abs{Df}^n/\abs{f-y_0}^n$ is integrable over $U$. We may hence apply \cite[Lemma 5.4]{Kangasniemi-Onninen_Heterogeneous} on $f - y_0$ to get that
	\[
		\log \abs{f - y_0} \in W^{1,n}_\loc(U).
	\]
	Similarly, an application of \cite[Lemma 6.1]{Kangasniemi-Onninen_Heterogeneous} on $f - y_0$ yields that for every $z \in U$ and a.e.\ $r \in (0, d(z, \partial U))$, we have
	\[
		\int_{\B^n(z, r)} \frac{J_f}{\abs{f - y_0}^n} \leq C_n r 	\int_{\partial \B^n(z, r)}\frac{\abs{Df}^n}{\abs{f - y_0}^n}.
	\]
	Now we use the previous estimate along with \eqref{eq:qrvalue} and H\"older's inequality to obtain
	\begin{multline*}
		\int_{\B^n(z, r)} \frac{\abs{Df}^n}{\abs{f-y_0}^n}
		\leq C_n K r \int_{\partial \B^n(z, r)} 
			\frac{\abs{Df}^n}{\abs{f-y_0}^n}
			+ \int_{\B^n(z, r)} \qrval\\
		\leq C_n K r \int_{\partial \B^n(z, r)} 
			\frac{\abs{Df}^n}{\abs{f-y_0}^n}
			+ \omega_n^\frac{p-1}{p} r^\frac{n(p-1)}{p} 
			\left(\int_{\Omega} \qrval^p\right)^\frac{1}{p}.
	\end{multline*}
	
	The above estimate lets us use \cite[Lemma 3.2]{Kangasniemi-Onninen_Heterogeneous} to obtain that
	\[
		\int_{\B^n(z, r)} \frac{\abs{Df}^n}{\abs{f-y_0}^n}
		\leq C r^\alpha,
	\]
	where $C > 0$ and $\alpha>0$ are independent of the choice of $z \in U$ and $r \in (0, d(z, \partial U))$. However, since we have $\abs{\nabla \log \abs{f - y_0}} \leq \abs{Df} / \abs{f - y_0}$, we now have a decay estimate
	\[
		\int_{\B^n(z, r)} \abs{\nabla \log \abs{f - y_0}}^n
		\leq C r^\alpha
	\]
	Hence, \cite[Lemma 3.3]{Kangasniemi-Onninen_Heterogeneous} yields that $\log \abs{f - y_0}$ is locally H\"older-continuous in $U$. This is a contradiction, since $x_0 \in U$ and $\lim_{x \to x_0} \log \abs{f(x) - y_0} = -\infty$. This concludes the proof.
\end{proof}

\section{The single value Reshetnyak's theorem}\label{sec:finishing_proof}

With Lemmas \ref{lem:fiber_hausd_measure} and \ref{lem:sense_preserving}, we now have the essential ingredients for the proof of Theorem \ref{thm:one_point_Reshetnyak}. We begin by proving discreteness, which we require in order to give a definition of the local index.

\begin{lemma}\label{lem:discreteness}
	Let $\Omega \subset \R^n$ be a connected domain, let $y_0 \in \R^n$, and let $f \in W^{1,n}_\loc(\Omega, \R^n)$ be a non-constant map that satisfies \eqref{eq:qrvalue} for a.e.\ $x \in \Omega$, where $K \in [1, \infty)$ and $\qrval \in L^p_\loc(\Omega)$ for some $p > 1$. Then $f^{-1} \{y_0\}$ is a discrete subset of $\Omega$.
\end{lemma}
\begin{proof}
	We wish to show that every $x \in f^{-1} \{y_0\}$ has a neighborhood that doesn't meet the rest of $f^{-1} \{y_0\}$. By restricting to a series of bounded subdomains, we may assume that $f \in W^{1,n}(\Omega, \R^n)$ and $\qrval \in L^p(\Omega)$.
	
	We assume towards contradiction that every neighborhood of a given $x_0 \in f^{-1} \{y_0\}$ meets $f^{-1} \{y_0\} \setminus x_0$. By Lemma \ref{lem:compact_preimages} and Corollary \ref{cor:small_preimages}, we may pick $\eps_0 > 0$ such that the $x_0$-component $U_0$ of $f^{-1} \B^n(x_0, \eps_0)$ is compactly contained in $\Omega$, and $m_n(U_0) < c$, with $c$ given by Lemma \ref{lem:weakly_sense_preserving}. Then the degree $\deg(f, U_0)$ is a finite integer; note that the finiteness can be immediately seen for example from Lemma \ref{lem:fiber_hausd_measure}, since $J_f$ is integrable.
	
	By our counterassumption, there exists $x_1 \in U_0$ such that $x_1 \in f^{-1}\{y_0\} \setminus x_0$. By Lemma \ref{lem:disjoint_preimages}, we may select $\eps_1 \in (0, \eps_0)$ such that if $U_1$ is the $x_0$-component of $f^{-1} \B^n(y_0, \eps_1)$, then $x_1 \notin U_1$. We let $U_{1, i}$, be the other components of $f^{-1} \B^n(y_0, \eps_1)$ contained in $U_0$, where $i \in I_1$. Then by Lemma \ref{lem:weakly_sense_preserving}, we have $\deg(f, U_{1,i}) \geq 0$ for all $i \in I_1$, since $m_n(U_{1,i}) \leq m_n(U_0) < c$. Moreover, since $x_1$ is necessarily in one of the sets $U_{1,i}$, we must have $\deg(f, U_{1,i}) > 0$ for at least one $i \in I_1$ by Lemma \ref{lem:sense_preserving}.
	
	Now, we may use both parts of \cite[Theorem 2.7]{Fonseca-Gangbo-book} to conclude that
	\[
		\deg(f, U_0) = \deg(f, U_1) + \sum_{i \in I_1} \deg(f, U_{1,i}) 
		> \deg(f, U_1)
	\]
	We can then repeat this procedure to find $U_2, U_3, \dots$ such that $\deg(f, U_1) > \deg(f, U_2) > \dots$. This is however a contradiction, as all degrees $\deg(f, U_i)$ are positive integers by Lemma \ref{lem:sense_preserving}. Hence, we conclude that our claim of discreteness holds.
\end{proof}

The discreteness leads to the definition of a local index for $f$ in $f^{-1} \{y_0\}$, which we recall here; see also e.g.\ \cite[Section I.4]{Rickman_book} or \cite[Definition 2.8]{Fonseca-Gangbo-book}.

\begin{defn}
	Let $\Omega \subset \R^n$ be a connected domain, let $y_0 \in \R^n$, and let $f \colon \Omega \to \R^n$ be a continuous map such that $f^{-1} \{y_0\}$ is a discrete set of $\Omega$. Let $x \in f^{-1} \{y_0\}$, and let $V_1, V_2$ be two neighborhoods of $x$ such that $\overline{V_1}\cap f^{-1} \{y_0\} = \overline{V_1}\cap f^{-1} \{y_0\} = \{x\}$. Then $y_0 \notin f(\partial V_i)$ for both $i \in \{1, 2\}$, and hence $\deg(f, y_0, V_i)$ is well defined for both such $i$. Moreover, we in fact have $\deg(f, y_0, V_1) = \deg(f, y_0, V_1 \cup V_2) = \deg(f, y_0, V_2)$ by using e.g.\ \cite[Theorem 2.7 (2)]{Fonseca-Gangbo-book}.
	
	Hence, if $V$ is a neighborhood of $x$ such that $\overline{V}\cap f^{-1} \{y_0\} = \{x\}$, then we get the same value of $\deg(f, y_0, V)$ regardless of the choice of $V$. This value is the \emph{local index of $f$ at $x$}, and is denoted by $i(x, f)$. 
\end{defn}

Under the assumptions of Theorem \ref{thm:one_point_Reshetnyak}, the discreteness of $f^{-1} \{y_0\}$ combined with Lemmas \ref{lem:compact_preimages} and \ref{lem:disjoint_preimages} implies that for all small enough $\eps > 0$, the $x$-component $U_\eps$ of $f^{-1}\B^n(y_0, \eps)$ is compactly contained in $\Omega$ and satisfies $\overline{U_\eps} \cap f^{-1} \{y_0\} = \{x\}$. Consequently, $i(x, f) = \deg(f, U_\eps)$ for any such $\eps$. Hence, the positive local index -part of Theorem \ref{thm:one_point_Reshetnyak} immediately follows from Lemma \ref{lem:sense_preserving}.

\begin{cor}\label{cor:local_index}
	Let $\Omega \subset \R^n$ be a connected domain, let $y_0 \in \R^n$, and let $f \in W^{1,n}_\loc(\Omega, \R^n)$ be a non-constant map that satisfies \eqref{eq:qrvalue} for a.e.\ $x \in \Omega$, where $K \in [1, \infty)$ and $\qrval \in L^p_\loc(\Omega)$ for some $p > 1$. Then for every $x_0 \in f^{-1} \{y_0\}$, the local index $i(x_0, f)$ is well defined and we have $i(x_0, f) > 0$.
\end{cor}

The final property to be deduced is the local openness property.

\begin{lemma}
	Let $\Omega \subset \R^n$ be a connected domain, let $y_0 \in \R^n$, and let $f \in W^{1,n}_\loc(\Omega, \R^n)$ be a non-constant map that satisfies \eqref{eq:qrvalue} for a.e.\ $x \in \Omega$, where $K \in [1, \infty)$ and $\qrval \in L^p_\loc(\Omega)$ for some $p > 1$. Then for every $x_0 \in f^{-1} \{y_0\}$ and every neighborhood $V$ of $x_0$, we have $y_0 \in \intr f(V)$.
\end{lemma}
\begin{proof}
	Suppose towards contradiction that $V$ is an open set containing $x_0 \in f^{-1} \{y_0\}$ and that $y_0 \notin \intr f(V)$. For all $\eps > 0$, we again use $U_\eps$ to denote the $x_0$-component of $f^{-1} \B^n(y_0, \eps)$. By Lemmas \ref{lem:compact_preimages} and \ref{lem:disjoint_preimages}, we may again pick an $\eps_0 > 0$ such that $\overline{U_{\eps_0}}$ is a compact subset of $\Omega$ and $\overline{U_{\eps_0}} \cap f^{-1} \{y_0\} = \{x\}$.
	
	Let $\eps < \eps_0$, in which case $\overline{U_{\eps}}$ is a compact subset of $\Omega$ and $\overline{U_{\eps}} \cap f^{-1} \{y_0\} = \{x\}$. By Corollary \ref{cor:local_index} we have $i(x_0, f) > 0$, so the definition of local index implies that $\deg(f, U_\eps) > 0$. In particular, by the definition of $\deg(f, U_\eps)$, this implies that $\deg(f, y, U_\eps) > 0$ for all $y \in \B^n(y_0, \eps)$. However, $\deg(f, y, U_\eps)$ can be non-zero only if $y \in f(U_\eps)$; see e.g.\ \cite[Theorem 2.1]{Fonseca-Gangbo-book}. We conclude that $f(U_\eps) = \B^n(y_0, \eps)$.
	
	By our counterassumption, $f(V)$ does not contain $\B^n(y_0, \eps) = f(U_\eps)$. We must hence have that $U_\eps \setminus V$ is non-empty. Now, $\overline{U_\eps} \setminus V$ with $\eps \in (0, \eps_0)$ form a decreasing family of non-empty compact subsets of $\Omega$. Hence, their intersection is non-empty. However, this is a contradiction, since
	\[
		\bigcap_{\eps \in (0, \eps_0)} (\overline{U_\eps} \setminus V) = \left( \bigcap_{\eps \in (0, \eps_0)}\overline{U_\eps} \right) \setminus V = \{x_0\} \setminus V = \emptyset.
	\]
	Our original claim therefore holds.
\end{proof}

\section{Other proofs}\label{sec_other_proofs}

We then prove the remaining results from the introduction, starting with Theorem \ref{thm:equivalence}. Before proceeding with the proof, we need to show that if $f$ satisfies \eqref{eq:qrvalue} with a higher integrable $\qrval \in L^p_\loc(\Omega), p > 1$, then $\abs{Df}$ also has higher integrability. The proof is by Gehring's lemma. Note that this result can also be used as an alternate proof of the local H\"older continuity of such $f$, although it does not yield a sharp exponent like the proof in \cite{Kangasniemi-Onninen_Heterogeneous}.

\begin{lemma}\label{lem:higher_integrability}
	Let $\Omega \subset \R^n$ be an open domain, let $f \in W^{1,n}(\Omega, \R^n)$, and let $K \geq 0$. Suppose that $f$ satisfies \eqref{eq:qrvalue}, where $\qrval \in L^p_\loc(\Omega)$ with $p > 1$. Then there exists $\beta > 1$ such that $f \in W^{1,\beta n}_\loc(\Omega)$.
\end{lemma}
\begin{proof}
	Let $Q$ be a cube with side length $r>0$, let $2Q$ denote the cube with the same center and and side length $2r$, and suppose that $\overline{2Q} \subset \Omega$. We may choose a cutoff function $\eta \in C^\infty_0(2Q)$ such that $\eta \geq 0$, $\eta \equiv 1$ on $Q$, and $\abs{\nabla \eta} \leq 4/r$. We use \eqref{eq:qrvalue} and a Caccioppoli inequality to obtain for every $c \in \R^n$ that
	\begin{multline*}
		\frac{1}{r^n} \int_{\Omega} \eta^n \abs{Df}^n
		\leq \frac{K}{r^n} \int_{\Omega} \eta^n J_f + \frac{1}{r^n} \int_{\Omega} \eta^n \qrval \abs{f-y_0}^n\\
		\leq \frac{C_1(n) K}{r^n} \int_{\Omega} \eta^{n-1} \abs{\nabla \eta} \abs{Df}^{n-1} \abs{f - c} + \frac{1}{r^n} \int_{\Omega} \eta^n \qrval \abs{f-y_0}^n.
	\end{multline*}
	In particular, we have
	\[
		\frac{1}{r^n} \int_{Q} \abs{Df}^n \leq  \frac{C_2(n) K}{(2r)^{n+1}} \int_{2Q} \abs{Df}^{n-1} \abs{f - c}
		+ \frac{1}{(2r)^n} \int_{2Q} 2^n \qrval \abs{f-y_0}^n.
	\]
	By a use of H\"older's inequality and the Sobolev-Poincar\'e inequality, the right choice of $c$ lets us estimate the first right-hand side term by
	\begin{multline*}
		\frac{1}{(2r)^{n+1}} \int_{2Q} \abs{Df}^{n-1} \abs{f - c}\\
		\leq \frac{1}{r} \left(\frac{1}{(2r)^{n}} \int_{2Q} \abs{Df}^{\frac{n^2}{n+1}} \right)^\frac{n^2 - 1}{n^2}
		\left(\frac{1}{(2r)^{n}} \int_{2Q} \abs{f - c}^{n^2} \right)^\frac{1}{n^2}\\
		\leq \frac{1}{r} \left(\frac{1}{(2r)^{n}} \int_{2Q} \abs{Df}^{\frac{n^2}{n+1}} \right)^\frac{n^2 - 1}{n^2}
		C_3(n) r \left(\frac{1}{(2r)^{n}} \int_{2Q} \abs{Df}^{\frac{n^2}{n+1}} \right)^\frac{n+1}{n^2}\\
		= C_3(n) \left(\frac{1}{(2r)^{n}} \int_{2Q} \abs{Df}^{\frac{n^2}{n+1}} \right)^\frac{n+1}{n}.
	\end{multline*}
	We conclude that
	\begin{multline*}
		\left( \frac{1}{r^n} \int_{Q} \abs{Df}^n \right)^\frac{n}{n+1}\\ 	\leq \frac{C_4(n) K}{(2r)^{n}} \int_{2Q} \abs{Df}^{\frac{n^2}{n+1}}
		+ \left(\frac{1}{(2r)^n} \int_{2Q} C_5(n) \qrval \abs{f-y_0}^n\right)^\frac{n}{n+1}.
	\end{multline*}
	Now, by a version of Gehring's lemma (see e.g.\ \cite[Proposition 6.1]{Iwaniec-GehringLemma}), it follows for small enough $\beta > 1$ that $\abs{Df}^n \in L^{\beta}_\loc(\Omega)$ if $\qrval\abs{f-y_0}^n \in L^{\beta}_\loc(\Omega)$. Such a $\beta > 1$ exists, since $\qrval \in L^{p}_\loc(\Omega)$ with $p > 1$, and since $f-y_0 \in L^q_\loc(\Omega)$ for all $q \in [1, \infty)$ by the Sobolev embedding theorem. The claim then follows.
\end{proof}

We then prove Theorem \ref{thm:equivalence}.

\begin{proof}[Proof of Theorem \ref{thm:equivalence}]
	It is immediately obvious that if $f$ is $K$-quasiregular, then we may pick $\qrval_{y_0} \equiv 0$ for every $y_0 \in \R^n$. The actual content of the theorem is hence the opposite direction.
	
	Suppose then that we have for every $y_0 \in \R^n$ a $\qrval_{y_0} \in L^p_\loc(\Omega)$ such that \eqref{eq:qrvalue} holds, where $p > 1$. The higher integrability for any single $\qrval_{y_0}$ implies by Lemma \ref{lem:higher_integrability} that $f \in W^{1,p'}_\loc(\Omega)$ with $p' > n$, which in turn implies that $f$ is almost everywhere differentiable (see e.g.\ \cite[Theorem 6.5]{Evans-Gariepy-book}). This is in fact all we need higher integrability for, and from now on it is enough to know that $\qrval_{y_0} \in L^1_\loc(\Omega)$ for every $y_0 \in \R^n$.
	
	We then let $x_0$ be a Lebesgue point of both $\abs{Df}^n$ and $J_f$ such that $f$ is differentiable at $x_0$. Due to the almost everywhere differentiability of $f$, this holds at almost every point $x_0 \in \Omega$. Our goal is to prove that $\abs{Df(x_0)} \leq K J_f(x_0)$.
	
	Since $f$ is differentiable at $x_0$, we also have
	\[
		\limsup_{x \to x_0} \frac{\abs{f(x) - f(x_0)}}{\abs{x - x_0}} = \abs{Df(x_0)}.
	\]
	If $\abs{Df(x_0)} = 0$, then $Df(x_0) = 0$ and hence $\abs{Df(x_0)} \leq K J_f(x_0)$. In particular, we may assume that $\abs{Df(x_0)} > 0$, and we're hence able to select $r_0 > 0$ such that $\B^n(x_0, r_0) \subset \Omega$ and
	\[
		\frac{\abs{f(x) - f(x_0)}}{\abs{x - x_0}} \leq 2\abs{Df(x_0)}
	\]
	for a.e.\ $x \in \B^n(x_0, r_0)$. Now, if $B_r = \B^n(x_0, r)$ with $r \in (0, r_0)$, then
	\[
		\abs{Df(x)}^n \leq K J_f(x) + 2 \abs{Df(x_0)}^n \abs{x - x_0}^n \qrval_{f(x_0)}
	\]
	for a.e.\ $x \in B_r$.
	
	Taking average integrals, we end up with
	\[
		\frac{1}{\omega_n r^n} \int_{B_r} \abs{Df(x)}^n
		\leq \frac{K}{\omega_n r^n} \int_{B_r} J_f(x)
		+ \frac{2\abs{Df(x_0)}^n}{\omega_n} \int_{B_r} \frac{\abs{x - x_0}^n}{r^n} \qrval_{f(x_0)}(x).
	\]
	Due to our choice of $x_0$ as a Lebesgue point of $\abs{Df}^n$ and $J_f$, the first two integrals converge to $\abs{Df(x_0)}^n$ and $K J_f(x_0)$ respectively as $r \to 0$. The last integral on the other hand converges to zero, since $\abs{x - x_0}^n < r^n$ in $B_r$ and since $\qrval_{f(x_0)} \in L^1_\loc(\Omega)$. Consequently, we get the desired
	\[
		\abs{Df(x_0)}^n \leq K J_f(x_0)
	\]
	in the limit, completing the proof.
\end{proof}

Next, we prove Theorem \ref{thm:small_K}.

\begin{proof}[Proof of Theorem \ref{thm:small_K}]
	Suppose that $y_0$ is a $(K, \qrval)$-quasiregular value of $f$, where we have assumed instead of $K \geq 1$ that $K \in [0, 1)$. Note that $f$ has a continuous representative by Lemma \ref{lem:higher_integrability}. Since $J_f \leq \abs{Df}$, we hence get
	\[
		\abs{Df}^n \leq (1 - K)^{-1} \qrval \abs{f - y_0}^n. 
	\]
	We then define $u \colon \Omega \to [0, \infty]$ by $u = \max(0, \log \abs{f - y_0}^{-2})$, and set $u_k = \min(u, k)$ for all $k \in \Z_{>0}$. Then we have $\abs{\nabla u_k} \equiv 0$ a.e.\ in the region where $\abs{f - y_0} \leq e^{-k/2}$, and 
	\[
		\abs{\nabla u_k}^n \leq \frac{\abs{Df}^n}{\abs{f - y_0}^n} \leq \frac{\qrval}{1-K}
	\]
	a.e.\ in the region where $\abs{f - y_0} > e^{-k/2}$. We may hence use Lemma \ref{lem:loglog_sobolev} and the higher integrability of $\qrval$ to conclude that either $f \equiv y_0$ in $\Omega$, or $u \in W^{1,p}_\loc(\Omega)$ for some $p > n$. In the latter case, we have that $u$ is locally essentially bounded by the Sobolev embedding theorem. Since $u(x) \to \infty$ when $d(x,f^{-1} \{y_0\}) \to 0$, this is only possible if $y_0 \notin f(\Omega)$.
\end{proof}

The last remaining result to prove is then Corollary \ref{cor:argument_principle}, which as we stated before is a standard degree theory argument now that we have Theorem \ref{thm:one_point_Reshetnyak}. We regardless give the proof for the convenience of the reader.

\begin{proof}[Proof of Corollary \ref{cor:argument_principle}]
	By our assumptions, we find $r > 0$ and $R > 0$ such that $\abs{f_2(x) - y_0} \geq 2r$ and $\abs{f_2(x) - f_1(x)} < r$ when $\abs{x} \geq R$. Notably, we also have $\abs{f_1(x) - y_0} \geq r$ when $\abs{x} \geq R$. In particular, $f_1^{-1}\{y_0\}$ and $f_2^{-1} \{y_0\}$ are both fully contained in $\B^n(0, R)$. 
	
	We then define a standard line homotopy $H \colon \R^n \times [0,1] \to \R^n$ between $f_1$ and $f_2$:
	\[
		H(x, t) = (1 - t) f_1(x) + t f_2(x).
	\]
	Since $f_i$ are continuous (again by \cite[Theorem 1.1]{Kangasniemi-Onninen_Heterogeneous} or Lemma \ref{lem:higher_integrability}), $H$ is also continuous. Moreover, if $\abs{x} \geq r$, we then have for every $t \in [0, 1]$ that
	\[
		\abs{H(x, t) - y_0} \geq \abs{f_2(x) - y_0} - \abs{H(x, t) - f_2(x)} > 2r - r > 0.
	\]
	In particular, $H(x, t) \neq y_0$ when $(x, t) \in \partial \B^n(0, R) \times [0, 1]$. Hence, we obtain the equivalence of topological degrees
	\[
		\deg(f_1, y_0, \B^n(0, R)) = \deg(f_2, y_0, \B^n(0, R));
	\]
	see e.g.\ \cite[Theorem 2.3 (2)]{Fonseca-Gangbo-book}.
	
	Furthermore, for each $i \in \{1, 2\}$, since $f_i^{-1} \{y_0\} \subset \B^n(0, R)$, we have that $f_i^{-1} \{y_0\}$ is a closed, bounded, discrete subset of $\R^n$ due to Theorem \ref{thm:one_point_Reshetnyak}. Hence, $f_i^{-1} \{y_0\}$ is a finite subset of $\B^n(0, R)$, and we may hence use e.g.\ \cite[Theorem 2.9 (1)]{Fonseca-Gangbo-book} to conclude that 
	\[
		\deg(f_i, y_0, \B^n(0, R)) = \sum_{x \in f_i^{-1}\{y_0\}} i(x, f_i).
	\]
	The claimed local index sum formula hence holds. The other claim also immediately follows by combining the local index sum formula with the positive local index part of Theorem \ref{thm:one_point_Reshetnyak}.
\end{proof}

\section{Examples}\label{sect:examples}

We begin with an example that the higher integrability of $\qrval$ is mandatory for Theorem \ref{thm:one_point_Reshetnyak}. 

\begin{ex}\label{ex:higher_int_needed}
	Let $A$ be any compact set of zero $n$-capacity such that $A \subset \B^n(0, 1)$. Then there exists a $\varphi \in W^{1,n}(\B^n(0, 1))$ such that $\varphi \geq 0$ and $\lim_{x \to x_0} \varphi(x) = \infty$ for every $x_0 \in A$. Indeed, by the definition of zero $n$-capacity, there exist smooth functions $\eta_i \in C^\infty_0(\B^n(0, 1))$ such that $\eta_i \geq 0$ everywhere, $\eta_i \geq 1$ on $A$, and $\norm{\nabla \eta_i}_{L^n} \leq 2^{-i}$. Since we also have $\norm{\eta_i}_{L^n} \leq C 2^{-i}$ by the Poincar\'e inequality, the infinite sum $\sum_{i \in \Z_{>0}} \eta_i$ then converges to a function in $W^{1,n}(\B^n(0, 1))$ which we may choose as our $\varphi$.
	
	We then select $\qrval = \abs{\nabla \varphi}^n \in L^1(\B^n(0, 1))$, and define
	\[
		f(x) = (\exp(-\varphi(x)), 0, \dots, 0).
	\]
	Now, $J_f \equiv 0$ everywhere, $f \in W^{1,n}(\B^n(0, 1))$, and $f$ satisfies $\abs{Df(x)}^n \leq \qrval(x) \abs{f(x)-0}^n$ for a.e.\ $x \in \B^n(0, 1)$. The image of $f$ has no interior points, so the third condition of Theorem \ref{thm:one_point_Reshetnyak} cannot hold for any $x_0 \in f^{-1}\{0\}$. The same is true of the second condition, since for every $r \in (0, 1)$ there are points in $\B^n(0, r)$ which are not in the image set of $f$, and hence we must have $\deg(f, U) = 0$ for every component $U$ of $f^{-1}\B^n(0, r)$ with $\overline{U} \subset \B^n(0, r)$.
	
	However, we also observe that $A \subset f^{-1} \{0\}$. Hence, any non-empty $A$ of zero $n$-capacity will yield a counterexample to the second and third conditions of Theorem \ref{thm:one_point_Reshetnyak} when only $\qrval \in L^1_\loc(\Omega)$. We may for instance use $A = \{0\}$, for which an explicit choice of $\varphi$ is given e.g.\ by $\varphi(x) = \log^\gamma(1 + \abs{x}^{-2})$ where $\gamma \in (0, (n-1)/n)$. Choosing an $A$ with an accumulation point will similarly yield a counterexample to the first condition; note that any countable $A$ has zero $n$-capacity by e.g. \cite[Lemma 2.8]{Heinonen-Kilpelainen-Martio_book}.
	
	We point out that this same method of producing counterexamples naturally fails if we require $\qrval \in L^{1+\eps}_\loc(\Omega)$. This is since we would then instead require that $A$ is of zero $(n + n\eps)$-capacity; however, even singletons $\{x_0\}$ don't have zero $p$-capacity if $p > n$.
\end{ex}

Next, we give another simple example which shows that, if $f \colon \Omega \to \R^n$ has a $(K, \qrval)$-quasiregular value at $y_0$ and $x_0 \in f^{-1} \{y_0\}$, then $f$ is not necessarily quasiregular in any neighborhood of $x_0$. This shows that the property of $f$ having a quasiregular value at $f(x_0)$ is different from $f$ being locally quasiregular near $x_0$.

\begin{ex}\label{ex:not_local_qr}
	We begin our construction by taking a collection of balls $B_i = \B^n(x_i, r_i), i \in \Z_{>0}$, where we have chosen $x_i = (2^{-i}, 0, 0, \dots, 0)$ and $r_i = 2^{-6i}$. It is clear that these balls are disjoint. We denote $B = \bigcup_i B_i$. We then define a map $f \colon \R^n \to \R^n$ by $f(x) = x$ when $x \in \R^n \setminus B$, and by
	\[
		f(x) = x + 2\left( r_i - \abs{x-x_i}\right) \frac{x_i}{\abs{x_i}} 
	\]
	when $x \in B_i$ for some $i \in \Z_{>0}$. Visually, the map shifts the center of each ball $B_i$ from $x_i$ to $x_i + (2r_i, 0, \dots, 0)$. In particular, this map has negative Jacobian in a region of positive measure in each $B_i$.
	
	We clearly have $\abs{Df} \leq C$ in $B_i$, with $C$ independent of $i$. Hence, $f$ is a Lipschitz map on $\R^n$. Moreover, we also have $\abs{f} \geq \abs{x_i} - r_i = 2^{-i} - 2^{-6i} > 2^{-2i}$. Hence, if we define
	\[
		\qrval \equiv 2 C^{n} 2^{2ni}
	\]
	in $B_i$, then $\abs{Df}^n + J_f^- \leq 2\abs{Df}^n \leq \abs{f}^n \qrval$ in $B_i$. The integral of $\qrval^2$ over $B_i$ is $(\omega_n 2^{-6ni}) (4 C^{2n} 2^{4ni}) = C' 2^{-2ni}$. Consequently, we have
	\[
		\sum_{i=1}^\infty \norm{\qrval}_{L^2(B_i)} < \infty.
	\]
	We then choose $\qrval \equiv 0$ in $\R^n \setminus B$, and conclude that $f$ has a $(1, \qrval)$-quasiregular value at $0$ with $\qrval \in L^2(\R^n)$. However, $f^{-1}\{0\} = \{0\}$, and there is no neighborhood of $0$ where $f$ is $K$-quasiregular, since every such neighborhood contains a ball $B_i$ which in turn contains a region where $f$ has negative Jacobian. 
\end{ex}

We then give a few more simple examples which illustrate how much the sets of quasiregular values can vary for different functions $f \in W^{1,n}_\loc(\Omega, \R^n)$. We leave the precise computations for the interested reader.

\begin{ex}\label{ex:not_in_sphere}
	Let $h \colon [0, \infty) \to \R$ be a piecewise linear function defined by
	\[
		h(t) = \begin{cases}
			t, & 0 \leq t < 1,\\
			1, & 1 \leq t < 2,\\
			t-1, & 2 \leq t.  
		\end{cases}
	\]
	We define $f \colon \R^n \to \R^n$ radially by
	\begin{equation}\label{eq:radial_f}
		f(x) = h(\abs{x})\frac{x}{\abs{x}}.
	\end{equation}
	Then $f \in W^{1,n}_\loc(\R^n, \R^n)$, and $f$ has a $(2, \qrval_{y_0})$-quasiregular value at every $y_0 \in \R^n \setminus \partial \B^n(0, 1)$ for some $\qrval_{y_0} \in L^\infty_\loc(\R^n)$. The map $f$ however has no quasiregular values at the points of $\partial \B^n(0, 1)$.
\end{ex}

\begin{ex}\label{ex:only_in_one_point}
	Similarly, one can find $0 = t_0 < t_1 < t_2 < \dots$ such that there exists a piecewise linear function $h \colon [0, \infty) \to \R$ as follows (see Figure \ref{fig:zigzag_h} for an illustration):
	\begin{itemize}
		\item $h(0) = 0$;
		\item $h(t_{2k - 1}) = 2^{k-1}$ and $h(t_{2k}) = 2^{-k}$ when $k \in \Z_{>0}$;
		\item $h'(t) \equiv 1$ on $(t_{2k}, t_{2k+1})$ and $h'(t) \equiv -1$ on $(t_{2k+1}, t_{2k+2})$ when $t \in \Z_{\geq 0}$.
	\end{itemize}
	If we use this $h$ to again define $f \colon \R^n \to \R^n$ by \eqref{eq:radial_f}, then $f \in W^{1,n}_\loc(\R^n, \R^n)$ and $f$ has a $(1, \qrval)$-quasiregular value at $0$ for some $\qrval \in L^\infty_\loc(\R^n)$. However, $f$ has no other quasiregular values $y_0 \in \R^n \setminus \{0\}$, as for instance every $f^{-1}\{y_0\}$ with $y_0 \in \R^n \setminus \{0\}$ contains points where $f$ has a negative local index.
	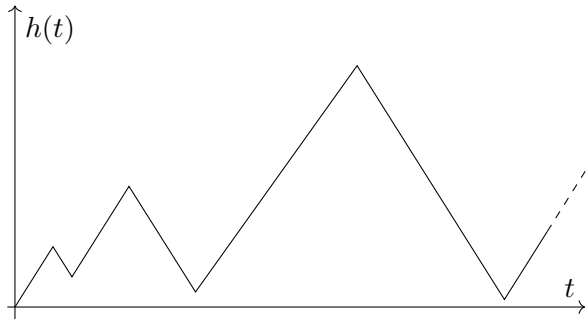
\begin{figure}[h]
		\begin{tikzpicture}[xscale=0.5, yscale=0.8]
			\draw[->] (-0.2, 0) -- (15, 0) node[anchor=south east]{$t$};
			\draw[->] (0, -0.2) -- (0, 5) node[anchor=north west]{$h(t)$};
			\draw (0,0) -- (1,1) -- (1.5, 0.5) -- (3, 2) -- (4.75, 0.25) -- 	(9, 4) -- (12.875, 0.125) -- (14, 1.25);
			\draw[dashed] (14, 1.25) -- (15, 2.25);
		\end{tikzpicture}
		\caption{\small An illustration of the function $h$: the peaks increase in height, while the valleys reach increasingly close to zero.}\label{fig:zigzag_h}
	\end{figure}
\end{ex}

Notably, Examples \ref{ex:not_in_sphere} and \ref{ex:only_in_one_point} show that under our given definitions, the set of quasiregular values of a continuous $W^{1,n}_\loc$-map is neither always open, nor always closed. Moreover, Example \ref{ex:not_in_sphere} shows that, in order for $f$ to be $K$-quasiregular, it is not enough to assume that almost every $y_0 \in \R^n$ is a $(K, \qrval_{y_0})$-quasiregular value of $f$.


\bibliographystyle{abbrv}
\bibliography{sources}

\end{document}